\documentclass{article}
\title{A Generalization of Brown's Construction for the Degree/Diameter Problem}
\author{Yawara ISHIDA}

\usepackage{bm}
\usepackage{amsthm}
\usepackage{amsmath}
\usepackage{amsfonts}
\usepackage{hyperref}

\newtheorem{Def}{Definition}
\newtheorem{Lem}{Lemma}
\newtheorem{Cor}[Lem]{Corollary}
\newtheorem*{Th}{Theorem}

\newcommand{\Z}{\mathbb Z}
\newcommand{\N}{\mathbb N}
\newcommand{\B}{\mathrm{B}}

\begin{document}
\maketitle
\begin{abstract}
The degree/diameter problem is the problem of finding the largest possible number of vertices $n_{\Delta,D}$ in a graph of given degree $\Delta$ and diameter $D$.
We consider the problem for the case of diameter $D=2$.
William G Brown gave a lower bound of the order of $(\Delta,2)$-graph.
In this paper, we give a generalization of his construction and improve the lower bounds for the case of $\Delta=306$ and $\Delta=307$.
One is $(306,2)$-graph with $88723$ vertices, the other is $(307,2)$-graph with $88724$ vertices.
\end{abstract}

\section{Introduction}

The {\it degree/diameter problem} is the problem of finding the largest possible number  $n_{\Delta,D}$ of vertices in a graph of given degree $\Delta$ and diameter $D$~\cite{MilSir2005, brown1966graphs}. 
Let $G$ be a graph with degree $\Delta$ ( $\Delta > 2$ ) and diameter $D$, then we have 
\[ |G| \leq n_{\Delta,D} \leq 1 + \Delta \sum_{k=1}^{D-1} (\Delta - 1)^k\]
where $|G|$ is the number of vertices of $G$. 
The right hand side of the above inequation is called {\it Moore bound}.
On the other hand, a lower bound of $n_{\Delta,D}$ for small degree and small diameter are available at \url{http://combinatoricswiki.org}. 
Especially for case of $D=2$ and large degree, there exists the general construction that gives a lower bound of $n_{\Delta,2}$, which is called {\it Brown's construction}~\cite{MilSir2005}.
Let $F_q$ be the finite field, where $q$ is a power of a prime.
Brown's construction gives the graph $\B(F_q)$ whose vertices are lines in $F_q^3$ and two lines are adjacent if and only if they are orthogonal. 
It follows that

\[|\B(F_q)| = q^2+q+1, \quad
\Delta(\B(F_q)) = q+1, \quad
D(\B(F_q))=2.
\]
The degree of each vertex of $\B(F_q)$ is $q+1$ or $q$. 
Among $q^2+q+1$ vertices, $q+1$ vertices are of degree $q$ and $q^2$ vertices are of degree $q+1$.
If $q$ is a power of $2$, there exists $(q+1,2)$-graph with $q^2+q+2$ vertices~\cite{journals/networks/ErdosFH80}.

In this paper, we generalize Brown's construction by replacing a field with a commutative ring, and search new records of the degree/diameter problem.

\section{Generalized Brown's Construction}
We give some definitions for generalized Brown's construction.
A {\it graph} $G=(V,E)$ consists of a set $V$ of {\it vertices} and a set $E \subset \{(v,w) \in V^2 | v \neq w \}$ of {\it edges}.
If $(v,w)$ is in $E$, it is said that $v$ and $w$ are {\it adjacent}, which is denoted by $v \sim w$.
The {\it order} $|G|$ of the graph is the number of vertices. 
The {\it neighbours} $N(v)$ of the vertex $v$ is a set of vertices which are adjacent to $v$.
The {\it degree} $\delta(v)$ of the vertex $v$ is the number of neighbours $| N(v) |$.  
The {\it degree} $\Delta(G)$ of the graph $G$ is the maximum degree of vertices, namely $\Delta(G)=\max\{\delta(v)|v\in V\}$.
The {\it distance} of the pair $(v,w)$ of vertices is the shortest path length between $v$ and $w$. 
The {\it diameter} $D(G)$ of the graph is the  maximum distance of all pairs of vertices.

Let $R$ be a commutative ring with unity. 
$R^*$ denotes the set of invertible elements of $R$.
$R^3$ is naturally seen as $R$-module. 
The addition and $R$-action are defined by coordinate-wise.
The {\it inner product} $\cdot: R^3 \times R^3 \Rightarrow R$ is defined as follows:
\[ (v_1,v_2,v_3) \cdot (w_1,w_2,w_3) = v_1 w_1 + v_2 w_2 + v_3 w_3 .\]
${\bm v}$ and ${\bm w}$ are {\it orthogonal} if and only if the inner product vanishes, namely ${\bm v} \cdot {\bm w} = 0$.
The {\it cross product} $\times: R^3 \times R^3 \Rightarrow R$ is defined as follows:
\[ (v_1,v_2,v_3) \cdot (w_1,w_2,w_3) = ( v_2 w_3 - v_3 w_2, v_3 w_1 - v_1 w_3, v_1 w_2 - v_2 w_1 ). \]

\begin{Def}
Let $R$ be a commutative ring with unity. The vertex set $V$ of the graph $\B(R)$ is 
\[ V = ( R^3 \setminus \{\bm v | \exists r \in R, r \cdot {\bm v} = {\bm 0} \} ) / \sim\]
where $\bm v \sim \bm w$ if and only if there exists $k \in R^*$ such that $k \cdot {\bm v} = {\bm w}$. The two vertices $[\bm v]$ and $[\bm w]$ are adjacent if and only if ${\bm v} \cdot {\bm w} = 0$.
\end{Def}

The definition of adjacency above is well-defined because the orthogonality does not depend on the selection of representatives. We call the above construction of a graph from a ring {\it generalized Brown's construction}. It is clear that the new construction coincides with Brown's one when the ring $R$ is a field. 

\begin{Lem}\label{Lem:regular}
Let $E$ be a Euclidean domain and $u$ be a prime element in $E$. 
If $E/(u^k)$ is a finite ring, then the degree of each vertex of $\B(E/(u^k))$ is $\Delta$ or $\Delta-1$, 
where $(u^k)$ is the principal ideal generated by $u^k$. 
\end{Lem}

\begin{proof}
It is clear that the degrees of vertices represented by $(1,0,0), (0,1,0), (0,0,1)$ are the same.
Let ${\bm v} = ([a],[b],[c])$ be a representative of any vertex where $a,b,c \in E$. 
If any element of $[a],[b],[c]$ is not invertible in $E/(u^k)$, 
there exist natural numbers $1 \leq l,m,n < k$ and some elements $a',b',c'$ in $E$ such that $a=u^l a', b=u^m b', c=u^n c'$. ${\bm v}$ is not a representative of vertices because the equation  $[u^{min(l,m,n)}] \cdot {\bm v} = ([0],[0],[0])$ holds. 
This is a contradiction.
Therefore, at least one of $[a],[b],[c]$ is invertible.
If $[a]$ is invertible, there exists one-to-one correspondence $\overline{U}: N([(1,0,0)])  \rightarrow N([{\bm v}])$ such that for all $[{\bm w}] \in N([1,0,0])$, $\overline{U}([\bm w]) = [ {}^t\!U^{-1} {\bm w} ]$ where $U$ is an invertible matrix defined as follows
\[
 U = \left(
 \begin{matrix}
  [a] & 0 & 0 \\
  [b] & 1 & 0 \\
  [c] & 0 & 1
 \end{matrix} \right)
\]
If ${\bm v} \cdot {\bm v} = 0$, then $\delta([\bm v]) = \delta([(1,0,0)] - 1$. If not so, $\delta([\bm v]) = \delta([(1,0,0)])$.
In the same way, if $[b]/[c]$ is invertible, then $\delta([\bm v])$ is $\delta([(0,1,0)])/\delta([(0,0,1)])$ or $\delta([(1,0,0)]/\delta([(0,0,1)]) - 1$.
Therefore, for all the vertex $[{\bm v}]$, 
\[ \delta([\bm v]) = \left\{ \begin{array}{ll}
    \delta([(1,0,0)]) & ({\bm v} \cdot {\bm v} \neq 0) \\
    \delta([(1,0,0)]) - 1 & (otherwise) 
    \end{array}\right.
\]

\end{proof}

\begin{Lem}\label{Lem:diameter}
Let $E$ be a Euclidean domain and $I$ be an ideal of $E$. The diameter of $\B(E/I)$ is 2.
\end{Lem}

\begin{proof}
For any two distinct vertices represented by $\bm v = ([v_1],[v_2],[v_3])$ and $\bm w=([w_1],[w_2],[w_3])$, consider the cross product $\bm v \times \bm w$. 
If $\bm v \times \bm w = \bm 0$, then $[v_i] \cdot \bm w = [w_i] \cdot \bm v$ for $i=1,2,3$. 
There exists $e \in E$ such that $I = (e)$ because any Euclidean domain is a principal ideal domain.
If $\gcd(v_1,v_2,v_3,e)$ is not a unity, where $\gcd$ is a greatest common divisor, 
there exists $e' \neq 1$ in $E$ such that $e=de'$. 
$\bm v$ is not a representative because $[e'] \cdot \bm v = \bm 0$. 
This is a contradiction. Therefore $d$ is a unity, namely $v_1$ and $v_2$, $v_3$, $e$ are coprime.
Then there exist $a, b, c, d \in E$ such that $ a v_1 + b v_2 + c v_3 + d e = 1$ in $E$.
Seeing this formula in $E/I$, we get $[a] [v_1] + [b] [v_2] + [c] [v_3] = [1]$.
\[ \bm v = [1] \cdot \bm v = ( [a] [v_1] + [b] [v_2] + [c] [v_3] ) \bm v = ( [a] [w_1] + [b] [w_2] + [c] [w_3] ) \bm w\]
means $[\bm v] = [\bm w]$, which is a contradiction to that two vertices are distinct, then $\bm v \times \bm w \neq \bm 0$. 
If $\bm v \times \bm w$ is a representative of vertex, $[\bm v \times \bm w]$ is adjacent to $[\bm v]$ and $[\bm w]$.
If $\bm v \times \bm w = ([k_1],[k_2],[k_3])$ is not a representative of vertex,
$\bm v \times \bm w = [\gcd(k_1,k_2,k_3)] \cdot \bm u$ and $\bm u$ is a representative of vertex. $[\bm u]$ is adjacent to $[\bm v]$ and $[\bm w]$.
\end{proof}

\begin{Th}\label{Th:main}
The following equations hold.
\begin{enumerate}
\item $ |\B(\Z_{p^k})| = p^{2k}+p^{2k-1}+p^{2k-2} $
\item $ \Delta(\B(\Z_{p^k})) = p^k + p^{k-1} $
\item $ D(\Z_{p^k}) = 2$
\end{enumerate}
\end{Th}

\begin{proof}
It is straightforward to show the formula of the order of $\B(\Z_{p^k})$.

\begin{eqnarray*}
|\B(\Z_{p^k})| & = & \frac{|\Z_{p^k}|^3 - |\{ mp | 0 \leq m < p^{k-1} \}|^3}{|\Z_{p^k}|-|\{ mp | 0 \leq m < p^{k-1} \}|} \\ 
& = & \frac{(p^k)^3 - (p^{k-1})^3}{p^k-p^{k-1}} = p^{2k}+p^{2k-1}+p^{2k-2}
\end{eqnarray*}

Using Lemma \ref{Lem:regular}, it is only enough to show that the degree of the vertex represented by $(1,0,0)$ satisfy the formula of the degree of $\B(\Z_{p^k})$. 
\begin{eqnarray*}
\Delta(\B(\Z_p^k)) & = & \delta([(1,0,0)]) = \frac{|\Z_{p^k}|^2 - |\{ mp | 0 \leq m < k \}|^2 }{|\Z_{p^k}|-|\{ mp | 0 \leq m < k \}|} \\
& = & \frac{(p^k)^2 - (p^{k-1})^2}{p^k-p^{k-1}} = p^k + p^{k-1}
\end{eqnarray*}

Using Lemma \ref{Lem:diameter}, we get $D(\B(\Z_p^k)) = 2$

\end{proof}

We search new records of the degree/diameter problem among graphs by generalized Brown's construction of $\Z_n$ where $2 \leq n \leq 10000$.
Using the above theorem, $\B(\Z_{17^2})$ has degree $306$ and diameter $2$ and $88723$ vertices.
It is a new record of $(306,2)$ of the degree/diameter problem because it cannot be obtained from ordinary Brown's construction.
The power of a prime less than $305=306 - 1$ is $293^1$ and the graph $\B(\Z_{293})$ obtained from ordinary Brown's construction of $293^1$ has $294=293+1$ degree and $86143=293^2+293+1$ vertices.
The old record of $306=294+12$ is $86156=86143+12$ obtained from $\B(\Z_{293})$ by duplicating vertices.
In the same way, the graph obtained from $\B(\Z_{17^2})$ by duplicating any one vertex, whose order is $88724$, is a new record of $(307,2)$ because the power of a prime less than $306=307-1$ is $293^1$.

\section*{Acknowledgements}
We wish to thank Ryosuke Mizuno, Nobuhito Tamaki, Masahito Hasegawa,Shin-ya Katsumata, and Sakie Suzuki.

\bibliographystyle{plain}
\bibliography{ref}

\end{document}